\documentclass[11pt]{amsart}
\usepackage{amssymb}
\usepackage{palatino}
\usepackage{tikz-cd}
\input amssym.def

\usepackage{amsmath, amsfonts}
\usepackage{amssymb}
\usepackage{amscd}
\usepackage[mathscr]{eucal}
\usepackage{palatino}
\usepackage{comment}

\newfont{\cyrr}{wncyr10}

\newcommand{\thmref}[1]{Theorem~\ref{#1}}
\newcommand{\corref}[1]{Corollary~\ref{#1}}

\newcommand{\lemref}[1]{Lemma~\ref{#1}}

\newcommand{\rmkref}[1]{Remark~\ref{#1}}

\newtheorem{thm}{Theorem}

\newtheorem{lem}[thm]{Lemma}
\newtheorem{cor}[thm]{Corollary}

\newtheorem{rmk}{Remark}[section] 
\newtheorem{defn}{Definition}

\def\GL{{\rm GL}}
\def\G{{\rm G}}

\def\({\left(}
\def\){\right)}
\def\[{\left[}
\def\]{\right]}

\def\N{\mathbb{N}}
\def\Z{\mathbb{Z}}

\def\C{\mathbb{C}}
\def\H{\mathbb{H}}
\def\Q{\mathbb{Q}}

\def\F{\mathbb{F}}

\def\cA{\mathcal{A}}
\def\cB{\mathcal{B}}

\def\cI{\mathcal{I}}
\def\cJ{\mathcal{J}}
\def\cN{\mathcal{N}}
\def\cO{\mathcal{O}}

\def\cD{\mathcal{D}}
\def\cR{\mathcal{R}}

\def\cS{\mathcal{S}}
\def\cT{\mathcal{T}}
\def\cX{\mathcal{X}}

\def\fp{\mathfrak{p}}

\def\a{\alpha}
\def\b{\beta}
\def\d{\delta}
\def\s{\sigma}
\def\z{\zeta}

\def\e{\epsilon}
\def\ve{\varepsilon}

\def\Zln{\mathbb{Z}/{\ell}^n\mathbb{Z}}

\def\uFl{\mathbb{F}^{\times}_{\ell}}
\def\det{\text{det}}
\def\tr{\text{tr}}
\def\pd{\partial}
\def\t{\otimes}

\newcommand{\tx}{\text}

\def\div{~|~}
\renewcommand{\mod}{{\, \rm mod \, }}

\usepackage{xcolor}

\title[largest prime factor of Fourier coefficients]{On the largest prime factor of non-zero Fourier coefficients 
of Hecke eigenforms}

\author{Sanoli Gun and Sunil L Naik}

\address{Sanoli Gun and Sunil L Naik \newline
The Institute of Mathematical Sciences, 
A CI of Homi Bhabha National Institute, 
CIT Campus, Taramani, 
Chennai 600 113, 
India.}
\email{sanoli@imsc.res.in}
\email{sunilnaik@imsc.res.in}

\begin{document}
	
\hfuzz 5pt



\maketitle

\begin{abstract}
Let $\tau$ denote the Ramanujan tau function. One is interested
in possible prime values of $\tau$ function. Since  $\tau$
is multiplicative and
$\tau(n)$ is odd if and only if $n$ is an odd square, we only need to consider
$\tau(p^{2n})$ for primes $p$ and natural numbers $n \ge 1$. 
This is a rather delicate question. In this direction, we show that
for any $\e>0$ and integer $n \geq 1$, the largest prime factor of $\tau(p^{2n})$, denoted by $P(\tau(p^{2n}))$, satisfies
$$
P(\tau(p^{2n})) ~>~ 
(\log p)^{1/8}(\log\log p)^{3/8 -\e}
$$
for almost all primes $p$.  This improves a recent
work of Bennett, Gherga, Patel and Siksek.
Our results are also valid for any non-CM
normalized Hecke eigenforms with integer Fourier coefficients.
\end{abstract}

\section{Introduction and Statements of Results}
The Ramanujan $\tau$ function is defined by
$$
\sum_{n=1}^{\infty} \tau(n)  X^n
~=~ 
X \prod_{n=1}^{\infty} (1- X^n)^{24}.
$$
Here the above identity can be interpreted
either as a formal power series over integers
 as well as  a function with specialization $X = e^{2\pi i z}$
with $z \in \H = \{ z \in \C : \Im(z) >0 \}$.

The motivation for our work comes from the following enigmatic question:
 when does  $\tau$ take a prime value? We note that $63001$
is the  smallest value of $n$ for which $\tau(n)$ is a prime  while
$157^{2206}$ is the largest known $n$
for which $\tau(n)$ is  prime (see \cite{Leh}, \cite{LR}).
Since $\tau$ is multiplicative and $|\tau(n)| = 1$ only when $n=1$ (see \cite{LR}),
we only need to consider prime powers. Hence the set of natural numbers 
$n$ for which $\tau(n)$ is a prime has natural density zero. It is expected 
that there are infinitely many such natural numbers \cite{LR}.

Studying largest prime factor of $\tau(n)$ is interesting only when
these values are non-zero. It is a folklore conjecture of Lehmer that  $\tau(n) \neq 0$ 
for all integers $n \geq 1$.  Serre \cite{Se81} showed that the set of natural 
numbers $n$ such that $\tau(n) \neq 0$ has positive density (see also \cite{Se85}).
Later Murty, Murty and Shorey \cite{MMSh} showed that there exists 
an absolute positive constant
 $c >0$ such that 
$$
|\tau(n)|  ~\geq~  (\log n)^c
$$
provided $\tau(n)$ is odd. It follows from Jacobi's triple product identity that $\tau(n)$ is 
odd exactly when $n$ is an odd square. 
Thus  it is prudent to consider $\tau(p^{2n})$ with $n>1$.
The primality of such  $\tau(p^{2n})$ is rather hard. In this paper,
we investigate whether $\tau(p^{2n})$ has a large prime factor
and more generally for non-CM normalized Hecke eigenforms
with integer Fourier coefficients.

Of course Deligne's  seminal  result
$$
|\tau(n)|  ~\leq~  d(n) n^{11/2},
$$
where $d(n)$ denotes the number of divisors of $n$,  provides us an
 immediate upper bound for the largest prime
 factor of $\tau(p^{2n})$.

Before moving further, we will fix few notions of densities. 
For a subset $S$ of primes, we shall define the lower and upper 
densities of $S$ to be
$$
\liminf_{x \rightarrow \infty} 
\frac{\#\{p \leq x: p \in S\}}{\pi(x)}
\phantom{m}\text{and}\phantom{m}
\limsup_{x \rightarrow \infty} 
\frac{\#\{p \leq x: p \in S\}}{\pi(x)}
$$
respectively.
Here $\pi(x)$ denotes 
the number of primes less than or equal to $x$.
If both lower and upper density of $S$ are equal,
say to $\cD$, then we say that $S$ has density $\cD$. We say a property $\cA$ 
holds for almost all primes if the set of primes having the property $\cA$ has density $1$.

For an integer $n$, let $P(n)$ denote the largest prime factor of $n$ with 
the convention that $P(0)= P(\pm 1)=1$. From now on $p, q, \ell$ will denote rational prime numbers.
In the case when $n$ is a prime $p$, Murty, Murty and Saradha \cite{MMS} 
proved that for any $\e > 0$ 
\begin{equation*}
	P(\tau(p)) > e^{(\log\log p)^{1-\e}}
\end{equation*}
for almost all primes $p$. In a recent work with Bilu \cite{BGN}, the authors  
proved that for any $\e>0$ 
\begin{equation}\label{low-1/8}
P(\tau(p)) > (\log p)^{1/8}(\log\log p)^{3/8 -\e}
\end{equation}
for almost all primes $p$. 

We know that the values of $\tau$ at prime powers 
satisfy the following recursion formula,
$$
\tau(p^m) ~=~ \tau(p) \tau(p^{m-1}) ~-~ p^{11} \tau(p^{m-2})
$$ 
for any prime $p$ and integer $m \geq 2$. Since $\tau(p^m)$ is a polynomial in $\tau(p)$,
one would hope that Diophantine techniques will help us to get some information 
regarding the largest prime factor of $\tau(p^m)$.
In a recent work \cite{BGPS}, Bennett, Gherga, Patel and Siksek 
proved that there exists an absolute constant $\alpha > 0$
such that for any prime $p$ with $\tau(p) \neq 0$ 
and any integer $m \geq 2$,
\begin{equation*}
\begin{split}
& P(\tau(p^m)) 
  ~>~ 
\alpha  \cdot \frac{\log\log (p^m)}{\log\log\log(p^m)}, \\
&	P(\tau(p^m)) 
    ~\geq~ 
    c(m) \cdot \log\log p,
\end{split}
\end{equation*}
where $c(m)$ is a positive constant which goes to zero 
as $m \to \infty$. 

In this article, using symmetric powers of a Galois representation attached to Ramanujan 
Delta function and divisibility properties of cyclotomic polynomials,
we prove the following theorem.
\begin{thm}\label{thm0}
Let $n \geq 1$ be an integer 
and $\e> 0$ be a real number. 
Then for almost all primes $p$, we have
$$
P(\tau(p^{2n})) ~>~ 
(\log p)^{1/8} (\log\log p)^{3/8-\e}.
$$
Further, if $q=P(2n+1)$ is sufficiently large, 
then the set of primes $p$ such that
$$
P(\tau(p^{2n})) ~>~ 
q^{-\epsilon} (\log p)^{1/8} (\log\log p)^{3/8}
$$
has positive lower density.
\end{thm}

\begin{rmk}
If one considers the product $\tau(p)\tau(p^2)$ 
instead of only $\tau(p)$, one has the following result of 
Luca and Shparlinski \cite{LS} which states that 
for almost all primes $p$,
\begin{equation}\label{33/31}
	P(\tau(p) \tau(p^2)) ~>~ (\log p)^{\frac{33}{31}+o(1)}.
\end{equation}
The exponent in the lower bound \eqref{33/31} was further refined 
to $13/11$ by Garaev, Garcia and Konyagin \cite{GGK}, 
but for infinitely many primes $p$.
\end{rmk}

\begin{cor}\label{Corpm}
	For any $\e>0$ and any integer $m \geq 1$,
	$$
	P(\tau(p^{m})) > (\log p)^{1/8}(\log\log p)^{3/8 -\e}
	$$
	for almost all primes $p$. 
\end{cor}

Let $f$ be a normalized cuspidal Hecke eigenform
of weight $k$ for $\Gamma_0(N)$ with trivial character 
lying in the newform space. The Fourier expansion of $f$ at infinity is given by
$$
f(z) ~=~ 
\sum_{n \geq 1} 
a_f(n) {e^{2\pi i nz}}~,
~~~z \in \H.
$$
We say that $f$ has complex multiplication (or $f$ is a CM form) if there exists an imaginary 
quadratic field $K$ such that $a_f(p)=0$ for all primes $p \nmid N$ 
which are inert in $K$. Otherwise $f$ is called a non-CM form.

Throughout the article, we assume that $a_f(n)$'s are rational integers. 
In this set up, we have the following result.
\begin{thm}\label{thm1}
Let $f$ be a non-CM normalized cuspidal Hecke eigenform
of weight $k$ and level $N$ having integer Fourier coefficients 
$\{a_f(n) ~:~ n \in \N\}$.
Let $q$ be an odd prime and $\e > 0$ be a real number. 
Then for almost all primes $p$, we have
$$
P(a_f(p^{q-1})) ~>~ 
(\log p)^{1/8} (\log\log p)^{3/8-\e}.
$$
Further, if $q$ is sufficiently large, then 
the set of primes $p$ such that
$$
P(a_f(p^{q-1})) ~>~  
q^{-\epsilon}(\log p)^{1/8} (\log\log p)^{3/8}
$$
has positive lower density.
\end{thm}

\begin{rmk}\label{rmkthm2-1}
We will derive \thmref{thm0} as a consequence of \thmref{thm1}.
Indeed \thmref{thm0} is valid for any non-CM normalized cuspidal Hecke eigenform 
$f$ with integer Fourier coefficients (see \S 5.3).
\end{rmk}

Conditionally under the Generalized Riemann Hypothesis (GRH), 
i.e., the Riemann Hypothesis for all Artin L-series, 
we have the following theorem.
\begin{thm}\label{thm2}
Suppose that GRH is true and $f$ be as in \thmref{thm1}. 
Let $q$ be an odd prime and $g$ be a real valued 
non-negative function satisfying $g(x) \to 0$ as $x \to \infty$.
Then we have
$$
P(a_f(p^{q-1})) ~>~  p^{g(p)}
$$
for almost all primes $p$. Further, there exists a positive constant $c$ depending only on $f$ 
such that the set of primes $p$ for which
$$
P(a_f(p^{q-1})) ~>~
c p^{1/14} (\log p)^{2/7}
$$
has lower density at least $1- \frac{2}{13(k-1)}$.
\end{thm}

Let us conclude this section with a brief summary of the article.
Background concerning the symmetric $n$-th power of 
Galois representations, modules over a commutative ring and cyclotomic polynomials 
will be provided in \S2, \S3 and \S 4 respectively.
We will derive some key results in \S3 and \S4 related to symmetric $n$-th power of matrices
and cyclotomic polynomials. In \S5, we prove two important lemmas regarding 
symmetric $n$-th power of Galois representations and then apply them to prove
the aforementioned theorems.
 
\medskip
 
\section{Preliminaries}

\subsection{Symmetric $m$-th power of a Galois representation}

Let $f$ be a normalized cuspidal Hecke eigenform of weight $k$ 
and level $N$ having integer 
Fourier coefficients $\{a_f(n): n \in \N \}$.
For any integer $d > 1$ and real number $x >0$, set
\begin{equation*}
\begin{split}
&\pi_{f,m}(x,d) 
~=~
 \#\{p \leq x ~:~ p \nmid dN,~ a_f(p^m) \equiv 0 ~(\mod d) \},\\
&\pi_{f,m}^*(x,d)
 ~=~
 \#\{p \leq x ~:~ p \nmid dN,~ a_f(p^m) \neq 0~,~ a_f(p^m) \equiv 0 ~(\mod d) \}.
	\end{split}
\end{equation*}
Let $\G  = \text{Gal}\(\overline{\Q}/\Q\)$ 
and for a prime $\ell$, let $\Z_\ell$ 
denote the ring of $\ell$-adic integers. 
By the work of Deligne \cite{De}, 
there exists a continuous representation
\begin{equation*}
	\rho_{d} ~:~ 
	{\G} ~\rightarrow~
	{\GL}_2\(\prod_{\ell | d} \Z_\ell\)
\end{equation*}
which is unramified outside the primes dividing $dN$. 
Further, if $p \nmid dN$, then we have
$$
\text{tr}\rho_{d}(\sigma_p) 
~=~ a_f(p) 
\phantom{mm}\text{and}\phantom{mm} 
\text{det}\rho_{d}(\sigma_p) 
~=~ p^{k-1},
$$
where $\sigma_p$ is a Frobenius element of $p$ in $\G$. 
Here $\Z$ is embedded 
diagonally in $\prod_{\ell | d} \Z_\ell$.
Let $\tilde{\rho}_{d}$ denote the reduction of $\rho_{d}$ modulo $d$ : 
\begin{equation*}
	\tilde{\rho}_{d} ~:~ {\G} 
	~\xrightarrow[]{\rho_{d}}~
	{\GL}_2\(\prod_{\ell | d} \Z_\ell\)
	~\twoheadrightarrow~
	{\GL}_2(\Z/ d\Z).
\end{equation*}
Also denote by $\tilde{\rho}_{d,m}$, the composition of 
$\tilde{\rho}_{d}$ with $Sym^m$, 
where $Sym^m$ denotes the symmetric $m$-th power map (see  \S 3) :
\begin{equation*}
\tilde{\rho}_{d,m} ~:~  {\G} 
~\xrightarrow[]{\rho_{d}}~ 
{\GL}_2\(\prod_{\ell | d} \Z_\ell\)  
~\rightarrow{}~
{\GL}_2(\Z/ d\Z)  
~\xrightarrow[]{Sym^m}~
{\GL}_{m+1}(\Z/ d\Z).
\end{equation*}
For $p \nmid dN$, we have (see \S 3, \lemref{trSymA} and \S4, Eq. \ref{PhPs})
\begin{equation*}
\tx{tr}\tilde{\rho}_{d,m}(\sigma_{p}) ~=~ a_f(p^m)~ (\mod d). 
\end{equation*}

Let $H_{d,m}$ be the kernel of $\tilde{\rho}_{d,m}$, 
$K_{d,m}$ be the subfield of $\overline{\Q}$ fixed by $H_{d,m}$ 
and ${\G}_{d,m} = \text{Gal}(K_{d,m}/\Q)$. 
Further suppose that $C_{d,m}$ is the subset of 
$\tilde{\rho}_{d,m}(\G)$ consisting of elements of trace zero. 
Let us set $\delta_{m}(d) = \frac{|C_{d,m}|}{|{\G}_{d,m}|}$.
For any prime $p\nmid dN$, the condition 
$a_f(p^m) \equiv 0 ~(\mod d)$ is equivalent to
the fact that $\tilde{\rho}_{d,m}(\sigma_p) \in C_{d,m}$, 
where $\sigma_{p}$ is a Frobenius element of $p$ in $\G$. 
Hence by the Chebotarev density theorem applied to $K_{d,m} / \Q$, 
we have
\begin{equation*}
\lim_{x \to \infty} \frac{\pi_{f,m}(x,d)}{\pi(x)}
~=~
 \frac{|C_{d,m}|}{|{\G}_{d,m}|}
~=~ \delta_{m}(d).
\end{equation*}

From the recent result of Thorner and Zaman (\cite{TZ}, Theorem 1.1), 
we have the following theorem.
\begin{thm}\label{CTZ}
Let $f$ be a non-CM normalized cuspidal Hecke eigenform
of weight $k$ and level $N$ with integer Fourier coefficients 
$\{a_f(n)~:~ n \in \N\}$. Then there exists an absolute constant $c>0$ 
such that for any natural number $d>1$ 
and for $\log x > cd^4 \log(dN)$, we have 
\begin{equation*}
\pi_{f,m}(x,d) 
~\ll~
 \d_{m}(d) \pi(x) .
\end{equation*}

\end{thm}

The following lemma follows from the work of Ribet \cite{Ri77} and Serre \cite{Se81} 
(also see \cite{MM84}).

\begin{lem}\label{Zx}
Let $f$ be as in \thmref{CTZ}. 
Then for any $\e> 0$, we have
\begin{equation*}
Z(x) 
~=~ \#\{p \leq x ~:~ a_f(p)=0\} 
~\ll_{\e}~
 \frac{x}{(\log x)^{3/2 -\e}}.
\end{equation*}
Further, suppose that GRH is true. Then we have
$$
Z(x) ~\ll~ x^{3/4}.
$$
\end{lem}
The next lemma relates vanishing of $a_f(p^m)$  
 to that of $a_f(p)$ (see \cite[Lemma 1]{MMSh}).
\begin{lem}\label{apn0}
Let $f$ be a normalized cuspidal Hecke eigenform of weight $k$ and level $N$ 
with integer Fourier coefficients. For all primes $p \nmid 6N$, 
$a_f(p^m) = 0$ if and only if $m$ is odd and $a_f(p) = 0$.
\end{lem}	

From the work of Lagarias and Odlyzko (\cite{LO}, Theorem 1.1) 
and applying \lemref{Zx}, 
we can deduce the following theorem (see \cite{MM84}, Lemma 5.3).
\begin{thm}\label{GRHpifmd}
Suppose that GRH is true and $f$ be as in \thmref{CTZ}. 
Then for $x \geq 2$, we have
\begin{equation}\label{GRHpifm}
\pi_{f,m}^{*}(x,d) 
~=~ \d_m(d) \pi(x) 
~+~ O\( \d_m(d) d^4 x^{1/2} \log(dNx)  \) 
~+~ O\(x^{3/4}\).
\end{equation}
The term $O(x^{3/4})$ is absent in \eqref{GRHpifm} if $m$ is even.
\end{thm}

We now state a lemma \cite[Prop. 2.4]{MM07} which
provides a lower bound for the Fourier coefficients at prime powers.
\begin{lem}\label{Lafpm}
Let $f$ be a normalized cuspidal Hecke eigenform of weight $k$ and level $N$. 
There exists a constant $c_1>0$ such that for any integer $m \geq 2$ 
and any prime $p \nmid N$ with $a_f(p^m) \neq 0$, we have
$$
|a_f(p^m)| 
~\geq~ 
\gamma_f(p,m)~ p^{\frac{k-1}{2}(m-c_1 \log m)},
$$
where 
\begin{equation*}
\gamma_f(p,m) ~=~
\begin{cases}
1 & \tx{if $m$ is even} \\
|a_f(p)| & \tx{if $m$ is odd}~~.
\end{cases}
\end{equation*}
\end{lem}

\subsection{The Sato-Tate conjecture}\label{ST}
The Sato-Tate conjecture is now a theorem due to the 
work of Barnet-Lamb, Clozel, Geraghty, 
Harris, Shepherd-Barron and Taylor 
(\cite[Theorem~B]{BGHT}, \cite{{CHT}, {HST}}). 
It states that for a non-CM normalized cuspidal Hecke eigenform $f$, 
the numbers 
$$
\lambda_f(p) ~=~ \frac{ a_f(p)}{2p^{(k-1)/2}}
$$ 
are equidistributed in the interval ${[-1, 1]}$ with respect 
to the Sato-Tate measure 
$$
d\mu_{ST} = (2/\pi)\sqrt{1- t^2}~dt.
$$ 
This means that for any $-1 \leq a \leq b \leq 1$, 
the density of the set of primes~$p$ 
satisfying ${\lambda_f(p) \in [a, b]}$ is 
\begin{equation*}
	\frac{2}{\pi}\int_a^b\sqrt{1-t^2}~ dt.
\end{equation*}
We have the following effective version of 
the Sato-Tate conjecture due to Thorner \cite{Th}.
\begin{thm}\label{EffST}
Let $f$ be a non-CM normalized cuspidal Hecke eigenform of weight $k$ and level $N$. 
Let $I \subseteq [-1, 1]$ be an interval. Then for $x \geq 3$, we have
$$
\#\{p \leq x ~:~ p \nmid N,~ \lambda_f(p) \in I\} 
~=~
 \mu_{ST}(I) \pi(x) 
 ~+~ O\(\pi(x) \frac{\log(kN \log x)}{\sqrt{\log x}}\).
$$
\end{thm}

\medskip

\section{Trace of symmetric $n$-th power of matrices}
Let $R$ be a non-zero commutative ring with identity. 
Let $M$ be a free $R$-module of rank $2$. 
For any integer $n \geq 1$, the tensor product 
$$
\cT^n(M)= \underbrace{M \t_{R} M\t_{R} \cdots \t_{R} M}_{n ~\tx{times}}
$$ 
is a free $R$-module of rank $2^n$. More precisely, if 
$\cB= \{e_1, e_2\}$ is an $R$-basis for $M$, then 
$$
\left\{v_1 \t v_2 \t \cdots \t v_n ~:~ v_i \in \cB , 1\leq i \leq n\right\}
$$
is an $R$- basis for $\cT^n(M)$.
Let $S_n$ denote the symmetric group on $n$ symbols. 
There is a natural $R$-linear left action of 
$S_n$ on $\cT^n(M)$ satisfying
$$
\s(v_1\t v_2 \t \cdots \t v_n) = 
v_{\s^{-1}(1)} \t v_{\s^{-1}(2)} \t \cdots \t v_{\s^{-1}(n)}
$$
for $\s \in S_n$  (see \cite[\S 10.4, \S 11.5, page 451]{DF}, \cite[\S 1.5, \S 1.6]{Se77}).

An element $w \in \cT^n(M)$ is called a symmetric $n$-tensor 
if $\s(w) = w,~ \forall~ \s \in S_n$. 
The set of all symmetric $n$- tensors
$$
\cS^n(M) = \{ w \in \cT^n(M)~:~ \s(w) = w ,~ \forall~\s \in S_n\}
$$
is a free $R$-module of rank $n+1$.

Let $\cX= \{1, 2\}$ and for each $\vec{j}= (j_1, j_2, \cdots ,j_n) \in \cX^n$, 
let $S_n(\vec{j})$ denote the orbit of $\vec{j}$ 
under the natural left action of $S_n$ on $\cX^n$. 
For each $1 \leq i \leq n+1$, set
$$
\vec{r}_{i} = 
(\underbrace{1, 1, \cdots, 1}_{n-i+1 ~ \tx{times}}, 
\underbrace{2,2,\cdots,2}_{i-1 ~ \tx{times}}) \in \cX^{n}
\phantom{mm}\tx{and}\phantom{mm} 
u_i = \sum_{\vec{j} \in S_n(\vec{r}_i)} e_{\vec{j}}~,
$$
where $e_{\vec{j}}$ is a shorthand notation for 
$e_{j_1} \t e_{j_2} \t \cdots \t e_{j_n}$. 
Then $\cB_n = \{u_i ~:~ 1 \leq i \leq n+1\}$
is an $R$-basis of $\cS^n(M)$.

Let $\tx{Aut}_R(M)$ be the group of all 
$R$-module automorphisms of $M$. 
Each $E \in \tx{Aut}_R(M)$ induces an automorphism 
$$
\tilde{E} : \cT^n(M) \longrightarrow \cT^n(M)
$$
such that 
$\tilde{E}(v_1 \t v_2 \t \cdots \t v_n) 
~=~ E(v_1)\t E(v_2) \t \cdots \t E(v_n)$. 
The restriction of $\tilde{E}$ to $\cS^n(M)$ denoted by  
$$
E_{\cS^n} : \cS^n(M) \longrightarrow \cS^n(M)
$$
is an automorphism of $\cS^n(M)$. Hence we have a group homomorphism 
\begin{align*}
	\phi : \tx{Aut}_R(M) & \longrightarrow \tx{Aut}_R(\cS^n(M)) \\
	E &\longmapsto E_{\cS^n}
\end{align*}
Let $[E]_{\cB}, [E_{\cS^n}]_{\cB_n}$ denote the matrices of 
$E, E_{\cS^n}$ with respect to the bases $\cB$ and $\cB_n$ respectively.
The map $\phi$ induces a group homomorphism 
from $\GL_2(R)$ to $\GL_{n+1}(R)$ given by
\begin{align*}
	\phi_{\cS^n} : \GL_2(R) &\longrightarrow \GL_{n+1}(R)\\	
	[E]_{\cB} &\longmapsto [E_{\cS^n}]_{\cB_n}
\end{align*}
Now onwards, we use the notation 
$Sym^n$ for the map $\phi_{\cS^n}$. 

\medskip

For any $\cA \in \GL_2(R)$, 
the following lemma provides an explicit expression for the entries of $Sym^n(\cA)$ 
in terms of the entries of $\cA$. 
\begin{lem}\label{SymA}
Let
\begin{equation*}
\cA ~=~
\begin{pmatrix}
a_{11} & a_{12} \\
a_{21} & a_{22}
\end{pmatrix}
\in \GL_2(R).
\end{equation*}
Then we have $Sym^n(\cA) = (c_{ij})$ with
\begin{equation*}
c_{ij} ~=~ 
\sum_{\vec{s} \in S_n(\vec{r}_j)}   a_{\vec{r}_i \vec{s}}~,
\phantom{m} i,j = 1,2,\cdots, (n+1),
\end{equation*}
where 
$a_{\vec{t} \vec{s}} 
= a_{t_1 s_1} a_{t_2 s_2} \cdots a_{t_n s_n}$ 
for $\vec{t}, \vec{s} \in \cX^n$.
\end{lem}

\begin{proof}
Let $E$ be the automorphism of $M$ whose matrix 
with respect to the basis $\cB$ is $\cA$. Then we have
$$
E(e_i) ~=~ a_{1i}e_1 + a_{2i} e_2 ~,~ i=1,2.
$$
Hence we get
\begin{equation*}
\begin{split}
E_{\cS^n}(u_j) 
&~=~ \sum_{\vec{s} \in S_n(\vec{r}_j)}  E(e_{\vec{s}})  \\
&~=~ \sum_{\vec{s} \in S_n(\vec{r}_j)} 
(a_{1 s_1} e_1 + a_{2 s_1} e_2) 
\t (a_{1 s_2} e_1 + a_{2 s_2} e_2) 
\t \cdots \t (a_{1 s_n} e_1 + a_{2 s_n} e_2)  \\
&~=~ \sum_{\vec{t} \in \cX^n} 
\sum_{\vec{s} \in S_n(\vec{r}_j)}  a_{\vec{t} \vec{s}} ~ e_{\vec{t}}.
\end{split}
\end{equation*}
Since $R$ is commutative, we note that
\begin{equation*}
\sum_{ \vec{s} \in S_n(\vec{r}_j) }  a_{\vec{w} \vec{s}}
~=~ \sum_{ \vec{s} \in S_n(\vec{r}_j) }  a_{\vec{t} \vec{s}}, 
~ \forall~ \vec{w} \in S_n(\vec{t}).
\end{equation*}
Thus we deduce that
\begin{equation*}
E_{\cS^n}(u_j)  
~=~ \sum_{i=1}^{n+1} 
\(\sum_{\vec{s} \in S_n(\vec{r}_j)} a_{\vec{r}_i \vec{s}} \) u_i.
\end{equation*}
Hence we have $Sym^n(\cA)= (c_{ij})$, where
$$
c_{ij} ~=~ 
\sum_{\vec{s} \in S_n(\vec{r}_j)}a_{\vec{r}_i \vec{s}}.
$$
\end{proof}

\begin{lem}\label{KerSym}
The kernel of $Sym^n$ is given by
\begin{equation*}
\tx{Ker}(Sym^n) ~=~
 \left\{ \lambda I_2 : \lambda \in R,~ \lambda^n = 1 \right\},
\end{equation*}
where $I_2 \in M_2(R)$ denotes the identity matrix.
\end{lem}

\begin{proof}
Suppose $\cA = (a_{ij}) \in \tx{Ker}(Sym^n)$,  then we have 
$Sym^n(\cA) = (c_{ij}) = I_{n+1}$, identity matrix of order $n+1$. 
From \lemref{SymA}, we get
\begin{equation*}
c_{11} ~=~ a_{11}^n = 1 ~,~~ 
c_{21} ~=~ a_{11}^{n-1} a_{21} = 0 ~,~~ 
c_{n (n+1)} ~=~ a_{22}^{n-1}a_{12} = 0 ~~\tx{and}~~
c_{(n+1)(n+1)} ~=~ a_{22}^n = 1.
\end{equation*}
This implies that $a_{21} = a_{12} = 0$. 
Again applying \lemref{SymA}, we obtain
$c_{22} = a_{11}^{n-1}a_{22} = 1$. Thus we get 
$a_{11} = a_{22}$, since $a_{11}^n = 1$. 
Hence we deduce that $\cA = a_{11} I_2$ with $a_{11}^n = 1$. 
Also we see that for any $\lambda \in R$ with $\lambda^n = 1$,
the matrix $\lambda I_2$ lies in the kernel of $Sym^n$. 
This completes the proof of the Lemma.
\end{proof}

\medskip

The following lemma describes the trace of $Sym^n(\cA)$ 
in terms of the trace and the determinant of $\cA$ (see \cite[page 77]{FH}).

\begin{lem}\label{trSymA}
For any integer $n \geq 2$ and $\cA = (a_{ij}) \in \GL_2(R)$, we have
\begin{equation*}
\tx{tr}(Sym^n(\cA)) 
~=~ \tx{tr}(A)^{\ve_{n+1}} F_{n+1}\(\tx{tr}(\cA)^2, \tx{det}(\cA)\).
\end{equation*}
\end{lem}

\begin{proof}
From \lemref{SymA}, we note that there exists a 
unique homogeneous polynomial 
$$
g(X_{11},X_{12}, X_{21},X_{22}) \in \Z[X_{11},X_{12}, X_{21},X_{22}]
$$ 
such that
$$
\tx{tr}(Sym^n(\cA)) ~=~ g(a_{11},a_{12}, a_{21},a_{22}),
$$
for any $\cA \in \GL_2(R)$ and any commutative ring $R$ with identity.
In order to determine $g$, without loss of generality, 
we can assume that $R= \C$. 
Since every matrix $\cA \in M_2(\C)$ has an eigen value in $\C$, 
without loss of generality, we can assume that $a_{21} = 0$ 
(by replacing $\cA$ with its conjugate, if necessary). 
Hence if $j < i$, then we have
$$
a_{\vec{r}_i \vec{s}} ~=~ 0,~
 \forall ~ \vec{s} \in S_n(\vec{r}_j).
$$
From \lemref{SymA}, we deduce that $c_{ij} = 0$ whenever $j < i$. 
Thus $Sym^n(\cA)$ is an upper triangular matrix with diagonal entries given by
$$
c_{ii} ~=~ \sum_{\vec{s} \in S_n(\vec{r}_i)}a_{\vec{r}_i \vec{s}} 
~=~ a_{\vec{r}_i \vec{r}_i} 
~=~ a_{11}^{n-i+1} a_{22}^{i-1}~,
\phantom{mm} 
i = 1, 2, \cdots, n+1.
$$
Hence by using \eqref{PhPs}, we get
\begin{equation*}
\begin{split}
\tx{tr}(Sym^n(\cA))
~=~ \sum_{i=1}^{n+1} c_{ii} 
~=~ \sum_{i=1}^{n+1} a_{11}^{n-i+1} a_{22}^{i-1} 
&~=~ \(a_{11}+a_{22}\)^{\ve_{n+1}} F_{n+1}\((a_{11}+a_{22})^2, a_{11} a_{22}\)\\
&~=~ \tx{tr}(\cA)^{\ve_{n+1}} F_{n+1}\( \tx{tr}(\cA)^2, \tx{det}(\cA)\).
\end{split}
\end{equation*}
Thus we conclude that 
$$
g(X_{11},X_{12}, X_{21},X_{22}) 
~=~ \(X_{11}+X_{22}\)^{\ve_{n+1}} 
    F_{n+1}\((X_{11}+X_{22})^2, X_{11} X_{22}-X_{21} X_{12}\).
$$
This completes the proof.
\end{proof}

\smallskip

\section{Cyclotomic polynomials and primitive divisors of $A^n-B^n$}

\smallskip

Let $K$ be a number field with the ring of integers $\cO_K$. 
Also let $A,B \in \cO_K \setminus\{0\}$. 

\begin{defn}
	A prime ideal $\fp$ of $K$ is called a primitive divisor
	of $A^n-B^n$ if $\fp \div A^n-B^n$ but
	$\fp \nmid A^m-B^m$ for $1 \leq m < n$. 
	Similarly, for $\gamma \in K\setminus\{0\}$,
	we say that $\fp$ is a primitive divisor of $\gamma^n-1$ 
	if $\nu_{\fp}(\gamma^n-1) \geq 1$ 
	but $\nu_{\fp}(\gamma^m-1) = 0$ for $1 \leq m < n$. 
	Here $\nu_{\fp}$ denotes the $\fp$-adic valuation.
\end{defn}

\begin{rmk}
	If $\fp$ is a primitive divisor of $A^n-B^n$ {\rm (}or $\gamma^n-1${\rm)}, then
	\begin{equation*}
		\cN_K(\fp) \equiv 1 (\mod n),
	\end{equation*}
	where $\cN_K$ denotes the absolute norm on $K$.
\end{rmk}

For any integer $n \geq 1$, the $n$-th cyclotomic polynomial $\Phi_n(X,Y) \in \Z[X,Y]$  is given by
$$
\Phi_n(X,Y) ~=~ 
\prod_{\substack{j=1\\ (j,n)=1}}^{n} \(X- \z_n^j Y\),
$$
where $\z_n = e^{\frac{2\pi i}{n}}$. For any integer $n \geq 3$, 
let $F_n(X,Y), \Psi_n(X,Y) \in \Z[X,Y]$ be given by
\begin{equation*}
	\Psi_n(X,Y) 
	~=~ 
	\prod_{\substack{1\leq j < n/2 \\ (j,n)=1}}  
	\(X- (\z_n^j+\z_n^{-j}+2) Y\) 
	\phantom{m} \tx{and} \phantom{m}
	F_n(X,Y)
	~=~ 
	\prod_{1\leq j < n/2} 
	\(X- (\z_n^j+\z_n^{-j}+2) Y\)	
\end{equation*}
respectively.
Then we have
\begin{equation}\label{PhPs}
	\begin{split}
		\Phi_n(X,Y) ~=~ \Psi_n\((X+Y)^2, XY\)
		\phantom{m}\tx{and}\phantom{m}
		\frac{X^n-Y^n}{X-Y} ~=~ (X+Y)^{\ve_n} F_n\((X+Y)^2, XY\),
	\end{split}
\end{equation}
where $\ve_n = 1$ if $n$ is even and $\ve_n = 0$ otherwise. 

We have the following lemma \cite[Lemma 4.10]{BGPS} concerning $\Psi_m(X,Y)$.
\begin{lem}\label{vPsi}
	Let $m = 5$ or $m \geq 7$. Let $u, v$ be coprime integers 
	and $p$ be a prime.
	Suppose  that $p^a \mid\mid \Psi_m(u,v)$, $a \geq 1$ an integer. 
	Then either
	$p \equiv \pm 1~ (\mod m)$ or $p^a \mid m$.
\end{lem} 

\medskip
 We need the following lemmas concerning $\Psi_q(X,Y)$ for any odd prime $q$.
\begin{lem}\label{DPsi}
Let $q$ be an odd prime and $\tilde{\Psi}_q(X) = \Psi_q(X,1)$. 
Then we have
\begin{equation*}
|\tilde{\Psi}_q(0)| ~=~ 1 
\phantom{m}~\tx{and}\phantom{m} 
 |D(\tilde{\Psi}_q)| ~=~ q^{\frac{q-3}{2}}, 
\end{equation*} 
where $D(\tilde{\Psi}_q)$ denotes 
the discriminant of the polynomial $\tilde{\Psi}_q$.
\end{lem}

\begin{proof}
We have
\begin{equation*}
\tilde{\Psi}_q(X) 
~=~ 
\prod_{j=1}^{\frac{q-1}{2}} \(X-(\z^j+\z^{-j}+2)\),
\end{equation*}
where $\z = \z_q$. Hence we have
\begin{equation*}
|\tilde{\Psi}_q(0)|
~=~  \prod_{j=1}^{\frac{q-1}{2}}|\z^j+\z^{-j}+2| 
~=~  \prod_{j=1}^{\frac{q-1}{2}} |\z^j+1|^2
~=~ \prod_{j=1}^{\frac{q-1}{2}} (\z^j+1)(\z^{-j}+1) 
~=~ \prod_{j=1}^{q-1} (\z^j+1) 
~=~ 1 .
\end{equation*}

Let $K = \Q(\zeta)$. Then the maximal real subfield of $K$ is 
$K^+ = \Q\(\zeta + \overline{\zeta}\)$. The discriminants of $K$ and $K^+$ 
are related by the following formula :
\begin{equation}\label{discDK}
|D_K| ~=~ |D_{K^+}|^2 \cdot N_{K^+}\(\mathscr{D}_{K/K^+}\),
\end{equation}
where $D_K, D_{K^+}$ denotes the discriminants of $K$ and $K^+$ 
respectively (see \cite[Ch. 3, \S2]{Neu}). Here $\mathscr{D}_{K/K^+}$ denotes
 the relative discriminant and $N_{K^+}$ denotes the absolute norm of $K^+$ over $\Q$. We have
$|D_K| ~=~ q^{q-2}$, $ |D_{K^+}| = |D(\tilde{\Psi}_q)|$ and  
$$
N_{K^+}\(\mathscr{D}_{K/K^+}\) ~=~ |N_{K^+} \( \(\zeta-\overline{\zeta}\)^2\)|
 ~=~ \prod_{i=1}^{\frac{q-1}{2}} |\z^i-\z^{-i}|^2.
$$
Note that
\begin{equation*}
\begin{split}
\prod_{i=1}^{\frac{q-1}{2}} |\z^i-\z^{-i}|^2 
&~=~ \prod_{i=1}^{\frac{q-1}{2}} |\z^{2i}-1|^2 
~=~ \prod_{i=1}^{\frac{q-1}{2}} \(\z^{2i}-1\) \(\z^{-2i}-1\) 
~=~ \prod_{i=1}^{q-1} \(\z^{2i}-1\) 
~=~ \prod_{i=1}^{q-1} \(\z^{i}-1\) 
~=~ q.
\end{split}
\end{equation*}
From \eqref{discDK}, we get $q^{q-2} = |D(\tilde{\Psi}_q)|^2 \cdot q$. 
This implies that $|D(\tilde{\Psi}_q)| = q^{\frac{q-3}{2}}$.
\end{proof}

\begin{lem}\label{pdPsi}
Let $q \ne \ell$ be distinct primes with $q$ odd and $u,v$ be integers such that 
\begin{equation*}
\Psi_q(u,v) ~\equiv~ 0 ~(\mod \ell), 
\phantom{m} (u,v)
 ~\not\equiv~ (0,0) ~(\mod \ell).
\end{equation*} 
Then we have
$$
\frac{\pd \Psi_q}{\pd X}(u,v) \cdot 
\frac{\pd \Psi_q}{\pd Y}(u,v) 
~\not\equiv~ 
0 ~(\mod \ell).
$$
\end{lem}
\begin{proof}
Clearly $v \equiv 0 ~(\mod \ell)$ implies that $u \equiv 0 ~(\mod \ell)$. 
Suppose $u \equiv 0 ~(\mod \ell)$, then we have 
$\Psi_q(u,v) \equiv \Psi_q(0,v) \equiv \pm v^{\frac{q-1}{2}} ~(\mod \ell)$
by applying \lemref{DPsi}. 
This implies that $v \equiv 0 ~(\mod \ell)$. Hence we deduce that
\begin{equation}\label{uv0l}
uv 
~\not\equiv~ 
0 ~(\mod \ell).
\end{equation}
Let $v'$ be an integer such that $v v' \equiv 1 ~(\mod \ell)$. 
Then we have
\begin{equation}\label{pdeq3}
	\Psi_q(u,v) 
	~\equiv~ 
	v^{\frac{q-1}{2}} \Psi_q(uv',1) ~(\mod \ell),
\end{equation}
\begin{equation}\label{pdeq4}
	\frac{\pd \Psi_q}{\pd X}(u,v) 
	~\equiv~ 
	v^{\frac{q-3}{2}} \frac{\pd \Psi_q}{\pd X}(uv',1)  ~(\mod \ell).
\end{equation}
Let $g(X) \equiv \Psi_q(X,1) ~(\mod \ell) \in \F_{\ell}[X]$, where $\F_\ell$ is a finite field with $\ell$
elements. Also let $w \equiv uv' ~(\mod \ell) \in \F_{\ell}$. 
Then from \eqref{pdeq3}, we have $g(w) = 0$ in $\F_{\ell}$ and from \lemref{DPsi}, 
we get $D(g) \neq 0$ in $\F_{\ell}$, since $\ell \neq q$. This implies that $g'(w) \neq 0$. 
Hence from \eqref{pdeq4}, we deduce that
\begin{equation}\label{delx0}
	\frac{\pd \Psi_q}{\pd X}(u,v) 
	~\not\equiv~ 0 ~(\mod \ell).
\end{equation}
The partial derivatives of $\Psi_q$ are given by
\begin{equation*}
\begin{split}
&\frac{\pd \Psi_q}{\pd X}(X,Y)
~=~ \sum_{j=1}^{\frac{q-1}{2}} ~\prod_{\substack{r=1 \\ r\neq j}}^{\frac{q-1}{2}} \(X-\lambda_rY\) 
\phantom{mm}\text{and}\phantom{mm}
\frac{\pd \Psi_q}{\pd Y}(X,Y) 
~=~ \sum_{j=1}^{\frac{q-1}{2}} \(-\lambda_j\) 
\prod_{\substack{r=1 \\ r\neq j}}^{\frac{q-1}{2}} \(X-\lambda_r Y\).
\end{split}
\end{equation*}
Here $\lambda_j = \z^j+\z^{-j}+2$ and $\z =\z_q$.
Hence we deduce that
\begin{equation}\label{DifPsiuv}
\begin{split}
X \frac{\pd \Psi_q}{\pd X}(X,Y)
~+~ Y \frac{\pd \Psi_q}{\pd Y}(X,Y) 
~=~ \sum_{j=1}^{\frac{q-1}{2}} \(X-\lambda_j Y\) 
\prod_{\substack{r=1 \\ r\neq j}}^{\frac{q-1}{2}} \(X- \lambda_r Y\) 
~=~ \frac{q-1}{2} \Psi_q(X,Y).
\end{split}
\end{equation}
From \eqref{uv0l}, \eqref{delx0} and \eqref{DifPsiuv}, we deduce that
\begin{equation*}
\frac{\pd \Psi_q}{\pd Y}(u,v) 
~\not\equiv~ 0 ~(\mod \ell).
\end{equation*}
\end{proof}

\section{Largest prime factor of Fourier coefficients at prime powers}

\subsection{Trace zero elements in the image of the symmetric $(q-1)$-th power 
of a Galois representation}

Let the notations be as before.
For any integer $n \geq 1$  and primes $q, \ell$ with $q$ odd, we have
\begin{equation*}\label{rhoClq-1}
|\tilde{\rho}_{\ell^n, q-1}(\G)|
~=~ 
\frac{|\tilde{\rho}_{\ell^n}(\G)|}{|\tx{Ker}(Sym^{q-1}) 
	 \cap \tilde{\rho}_{\ell^n}(\G)|}
\phantom{mm}\tx{and}\phantom{mm}
|C_{\ell^n, q-1}| 
~=~ 
\frac{|D_{\ell^n, q-1}|}{|\tx{Ker}(Sym^{q-1}) 
		\cap \tilde{\rho}_{\ell^n}(\G)|},
\end{equation*}
where 
$$
D_{\ell^n,q-1} 
~=~
\{ A \in \tilde{\rho}_{\ell^n}(\G) ~:~ \Psi_{q}\(\tr(A)^2, \det(A)\) = 0\}.
$$
Hence we have
\begin{equation}\label{rhoq-1=}
\d_{q-1}(\ell^n) 
~=~
\frac{|C_{\ell^n, q-1}|}{|\tilde{\rho}_{\ell^n, q-1}(\G)|}
~=~
\frac{|D_{\ell^n, q-1}|}{|\tilde{\rho}_{\ell^n}(\G)|}.
\end{equation}
We know from the works of Carayol \cite{Ca}, 
Momose \cite{Mo}, Ribet \cite{Ri75,Ri85}, Serre \cite{Se76}
and Swinnerton-Dyer \cite{Sw} that the image of $\rho_\ell$ is open in $\GL_2(\Z_\ell)$ and 
\begin{equation}\label{Gimg}
\rho_{\ell}(\G)
~\subseteq~
\{ A \in \GL_2(\Z_{\ell}) 
	~:~ \det (A) \in (\Z^{\times}_{\ell})^{k-1}\}
\end{equation}
and equality holds when $\ell$ is sufficiently large (see Eq. 9, Lemma 1 and Theorem 4 of \cite{Sw}).
Hence for any $n \in \N$, we have
\begin{equation}\label{imglnGdet}
\tilde{\rho}_{\ell^n}(\G)
~\subseteq~
\{ A \in \GL_2(\Z/\ell^n \Z) 
~:~ \det (A) \in ((\Z/\ell^n \Z)^{\times})^{k-1}\}
\end{equation}
and equality holds when $\ell$ is sufficiently large.
Let $d = \tx{gcd}(\ell-1, k-1)$, then we have
\begin{equation}\label{Tl}
|\tilde{\rho}_{\ell}(\G)| 
~=~
 \frac{(\ell^2-1)(\ell^2-\ell)}{d} 
 \phantom{mm}\tx{and}\phantom{mm} 	
|\tilde{\rho}_{\ell^n}(\G)| 
~=~
 \ell^{4(n-1)} |\tilde{\rho}_{\ell}(\G)|
\end{equation}
for all sufficiently large $\ell$ (see \cite[\S3]{GM}).
 We also have  
\begin{equation}\label{Tlall}
|\tilde{\rho}_{\ell}(\G)| 
~\gg~
\ell^4,
\end{equation}
where the implied constant depends only on $f$ 
(see \cite[Lemma 5.4]{MM84}, \cite[Prop. 17]{Se81}).

\begin{lem}\label{deq-1l}
Let $q, \ell$ be primes with $q$ odd. Then 
$\d_{q-1}(\ell) = 0$ unless $\ell \equiv 0, \pm 1 ~(\mod q)$
and 
$$
\d_{q-1}(\ell) \ll \frac{q}{\ell},
$$
where the implied constant depends only on $f$. Also we have
\begin{equation*}
\d_{q-1}(\ell) 
~=~
\begin{cases}
& \frac{q-1}{2} \frac{1}{\ell-1}, 
\phantom{mm}\tx{if}\phantom{mm} 
\ell \equiv 1~ (\mod q) \\
& \frac{q-1}{2}   \frac{1}{\ell+1}, 
\phantom{mm}\tx{if}\phantom{mm} 
\ell \equiv -1~ (\mod q) \\
& \frac{q}{q^2 -1}, \phantom{mmm}\tx{if} \phantom{mm} \ell = q 
\end{cases}
\end{equation*}
for all sufficiently large $\ell$.
\end{lem}

\begin{proof}
Set
\begin{equation*}
\begin{split}
  D'_{\ell, q-1} 
  = \{A \in \GL_2(\F_{\ell}) ~:~ \det(A) \in (\F^{\times}_{\ell})^{k-1},~ \Psi_{q}\(\tr(A)^2, \det(A)\) = 0 \}.
 \end{split}
\end{equation*}
Then we have $D_{\ell, q-1} \subseteq D'_{\ell, q-1}$ and
 equality holds if $\ell$ is sufficiently large (see Eq. \eqref{imglnGdet}).

We first suppose that $\ell \neq 2$. The conjugacy classes of $\GL_2(\F_{\ell})$ are one of the following four types:
\begin{equation*}
\begin{split}
&\a_a 
~=~
\begin{pmatrix}
a & 0 \\
0 & a
\end{pmatrix},
\phantom{mmmmmmm}
\tilde{\a}_a 
~=~
\begin{pmatrix}
	a & 1 \\
	0 & a
\end{pmatrix}, \\
& \a_{a,b}
~=~
\begin{pmatrix}
	a & 0 \\
	0 & b
\end{pmatrix},~ a ~\neq~ b,
\phantom{mm}
\b_{b,c}
~=~
\begin{pmatrix}
	c & \eta b \\
	b & c
\end{pmatrix},
\end{split}
\end{equation*}
where $a, b \in \F^{\times}_{\ell},~ c \in \F_{\ell}$ 
and $\eta$ is a fixed element of $ \F^{\times}_{\ell}$ so that
 $\{1, \sqrt{\eta}\}$ is a basis for $\F_{\ell^2} = \F_{\ell}(\sqrt{\eta})$ over $\F_{\ell}$. 
The number of elements in these classes are $1, \ell^2-1, \ell^2+\ell$ and $\ell^2-\ell$ 
respectively (see page 68 of \cite{FH}). Also note that $\a_{a, b},
\a_{b, a}$ as well as $\b_{b,c}, \b_{-b,c}$ are conjugates of each other.

Suppose that $\ell \neq q$. If $\a_a \in D'_{\ell, q-1}$, then we have
$$
a^2 ~\in~ (\F^{\times}_{\ell})^{k-1} 
\phantom{mm}\tx{and}\phantom{mm} 
\Psi_q(4a^2,a^2) ~=~ 0.
$$
Since 
$\Psi_q(4a^2,a^2)
= (a^2)^{\frac{q-1}{2}} \Psi_q(4,1)
= a^{q-1} \Phi_q(1,1)
= a^{q-1}q$, 
we get $a = 0$, a contradiction. 
Hence no conjugacy class of type $\a_a$ belongs to $D'_{\ell, q-1}$. 
By the same arguments, $D'_{\ell, q-1}$ contains 
no element of $\GL_2(\F_{\ell})$ which is conjugate to 
$\tilde{\a}_a$ for some $a \in \uFl$.

Suppose $\a_{a,b} \in D'_{\ell, q-1}$, then we have
\begin{equation*}\label{aad}
ab ~\in~ (\uFl)^{k-1} 
\phantom{mm}\tx{and}\phantom{mm} 
\Psi_q((a+b)^2, ab) ~=~ 0.
\end{equation*}
Let $t = b a^{-1}$, 
where $a^{-1}$ denotes the multiplicative inverse of $a \in \F_{\ell}^{\times}$. 
Then we have
$$
a^2 t ~\in~ (\uFl)^{k-1} ~,
\phantom{mm} 
\Psi_q(a^2(1+t)^2, a^2t) ~=~ 0 
\phantom{mm}\tx{and}\phantom{mm}
 t ~\neq~ 1 .
$$
By using the fact that 
$t^q-1 = (t-1) \Psi_q((1+t)^2,t)$, we get
\begin{equation}\label{tql}
t^q ~=~ 1, \phantom{mm} t ~\neq~ 1.
\end{equation} 
The above equation \eqref{tql} has a solution in $\F_{\ell}$ 
only when $q \mid \ell-1$ and in that case
it has exactly $q-1$ solutions. 
We claim that for any $t \in \uFl$,
\begin{equation*}
\#\{ a \in \uFl ~:~ a^2 t \in  (\uFl)^{k-1}\} = \frac{\ell -1}{d},
\end{equation*}
where $d = (\ell-1, k-1)$.
Let $g$ be a primitive root modulo $\ell$ i.e., $\uFl =~ <g>$ 
and write $a = g^j,~ t = g^r$. 
We note that $(\uFl)^{k-1} = (\uFl)^d$ and 
$a^2 t \in  (\uFl)^{k-1}$ is equivalent to the fact that
\begin{equation}\label{a2t}
2j + r ~\equiv~ sd~ (\mod~ \ell-1) \phantom{m}\tx{for some } s \in \F_{\ell}.
\end{equation}
It implies that $2j \equiv -r ~(\mod d)$ and this congruence has a unique solution
since $(2,d) = 1$. Hence \eqref{a2t} has exactly $\frac{\ell-1}{d}$ solutions. This proves our claim.
Thus if $q \mid \ell-1$, then $D'_{\ell,q-1}$ contains 
exactly $\frac{(q-1)}{2}\frac{(\ell-1)}{d}$ conjugacy classes 
of the form $\a_{a,b}$. If $ q\nmid \ell -1$, there are no
conjugacy class  of the form $\a_{a,b}$ which belongs to $D'_{\ell,q-1}$.

Suppose  $\b_{b,c} \in D'_{\ell,q-1}$, then we have
\begin{equation}\label{bbc}
c^2-\eta b^2 ~\in~ (\uFl)^{k-1} 
\phantom{mm}\tx{and}\phantom{mm} 
\Psi_q(4c^2, c^2-\eta b^2) ~=~ 0.
\end{equation}
Note that $c \neq 0$. Otherwise, we have 
$$
\Psi_q(4c^2, c^2-\eta b^2) 
~=~ \pm(-\eta b^2)^{\frac{q-1}{2}} ~=~ 0.
$$
This implies that $b = 0$, a contradiction.
Set $t = bc^{-1}$. From \eqref{bbc}, 
we get 
$$
\Psi_q(4c^2, c^2(1-\eta t^2)) 
~=~ c^{q-1}\Psi_q(4, 1-\eta t^2) ~=~ 0.
$$
Hence we get $\Psi_q(4, 1-\eta t^2) = 0$.
Note that 
$1-\eta t^2 = (1+\sqrt{\eta}t) (1-\sqrt{\eta}t)$
 in $\F_{\ell^2}$.
 Hence we have 
$$
\Psi_q(4,1-\eta t^2) 
~=~ \Psi_q(2^2,(1+\sqrt{\eta}t) (1-\sqrt{\eta}t))
~=~ \frac{(1+\sqrt{\eta}t)^q - (1-\sqrt{\eta}t)^q}{2 \sqrt{\eta}t}.
$$
This implies that 
$(1+\sqrt{\eta}t)^q - (1-\sqrt{\eta}t)^q = 0$. 
By noting that $1- \sqrt{\eta}t \neq 0$, we obtain
$$
\(\frac{1+\sqrt{\eta}t}{1-\sqrt{\eta}t}\)^{q} ~=~ 1.
$$
Hence we get $q \mid \ell^2- 1$. Further note that
$q \nmid \ell-1$ as otherwise
$$
\frac{1+\sqrt{\eta}t}{1-\sqrt{\eta}t} ~\in~ \uFl 
$$
and hence $\sqrt{\eta} \in \uFl$, a contradiction. 
Thus we have $q \mid \ell+1$.
 We claim that for every solution $x_0$ of $X^q =1$ in $ \F_{\ell^2}$,
 there exists a unique $t \in \F_{\ell}$ such that $x_0= \frac{1+\sqrt{\eta}t}{1-\sqrt{\eta}t}$. 
 Write $x_0 = u +\sqrt{\eta}v \in \F_{\ell^2}$ with $u, v \in \F_{\ell}$. Then we 
 get $(u^2-\eta v^2)^q = 1$. Moreover, $u^2-\eta v^2 \in \uFl$ and $q \nmid (\ell-1)$ implies that
 $u^2-\eta v^2=1$. Note that $u \neq -1$, otherwise we get $v=0$ and hence 
 we obtain $x_0^q = (-1)^q =-1$, a contradiction. We see that $t=v (u+1)^{-1}$ is the 
 unique element in $\F_{\ell}$ such that
 $$
x_0 ~=~ u +\sqrt{\eta}v ~=~ \frac{1+\sqrt{\eta}t}{1-\sqrt{\eta}t}.
 $$
As before for any $t \in \uFl$, we have
\begin{equation*}
\#\{ c \in \uFl ~:~ c^2(1-\eta t^2)  \in  (\uFl)^{k-1}\} 
~=~ \frac{\ell -1}{d},
\end{equation*}
where $d = (\ell-1, k-1)$.
Thus we deduce that $D'_{\ell, q-1}$ contains exactly $\frac{(q-1)}{2}\frac{(\ell-1)}{d}$ 
conjugacy classes of type $\b_{b,c}$ when $q \mid \ell + 1$. When
$q \nmid \ell +1$, there are no such conjugacy classes in $D'_{\ell, q-1}$.

Now suppose that $\ell = q$. Arguing as before, we deduce that
$D_{q, q-1}^{'}$ contains $\frac{q -1}{d}$ conjugacy classes of type $\a_a$ 
and $\frac{q -1}{d}$ conjugacy classes of type $\tilde{\a}_a$. Further
arguments as before will imply that $D_{q, q-1}^{'}$ does not contain any 
element which is conjugate to $\a_{a,b}$ or $\b_{b,c}$. 

When $\ell =2$, $\GL_2(\F_{2})$ is a union of $3$ conjugacy classes whose representatives are  given by
\begin{equation*}
\begin{pmatrix}
1 & 0 \\
0 & 1
\end{pmatrix},
~~~
\begin{pmatrix}
	1 & 1 \\
	0 & 1
\end{pmatrix}
\phantom{m}\tx{and}\phantom{m}
\begin{pmatrix}
	1 & 1 \\
	1 & 0
\end{pmatrix}.
\end{equation*} 
The number of elements in these classes are $1, 3$ and $2$ respectively. 
Arguing as before, we can show that $|D'_{2,2}|=2$ and $|D'_{2,q-1}| =0$ for $q \geq 5$.
Hence we conclude that
\begin{equation}\label{dlq-1}
|D_{\ell, q-1}^{'}| 
~=~ 
\begin{cases}
& \frac{(q-1)(\ell-1)}{2d} (\ell^2+\ell) ,
\phantom{mm} \tx{if}\phantom{mm} 
\ell \equiv 1 ~(\mod q) \\
& \frac{(q-1)(\ell-1)}{2d} (\ell^2-\ell) ,
\phantom{mm} \tx{if}\phantom{mm} 
\ell \equiv -1 ~(\mod q) \\
& \frac{q^2 (q-1)}{d}  ,
\phantom{mmmmmmm} \tx{if}\phantom{mm} 
\ell = q \\
& 0, \phantom{mmmmmmmmmm}\tx{otherwise}.
\end{cases}
\end{equation}
Now \lemref{deq-1l} follows from \eqref{rhoq-1=}, \eqref{Tl}, \eqref{Tlall} and \eqref{dlq-1}.
\end{proof}

\begin{lem}\label{dq-1lm}
For any integer $n \geq 2$ and primes $\ell, q$ with $q$ odd, we have
$$
\d_{q-1}(\ell^n) 
~\ll~
 \frac{1}{\ell^{n-1}} \d_{q-1}(\ell),
$$
where the implied constant depends only on $f$.
We also have
$$
\d_{q-1}(\ell^n) ~=~ \frac{1}{\ell^{n-1}} ~\d_{q-1}(\ell)
$$
if $\ell \neq q$ and $\ell$ is sufficiently large.
Further $\d_{q-1}(q^n) = 0$ for $q \geq 5$.
\end{lem}

\begin{proof}
Note that
\begin{equation}\label{Dlneqlift}
D_{\ell^n, q-1} ~\subseteq~  
\left\{A \in \GL_2\(\Zln\) ~:~ A~(\mod \ell^{n-1}) \in D_{\ell^{n-1}, q-1},
~ \Psi_{q}(\tr(A)^2, \det(A)) = 0 \right\}
\end{equation}
and the equality holds  if $\ell$ is sufficiently large. 
To see the equality, note that if $\ell$ is sufficiently large, 
then we have (see \eqref{imglnGdet}) 
\begin{equation}\label{imglnlift}
\begin{split}
\tilde{\rho}_{\ell^n}(\G)
~=~ \{ A \in \GL_2\(\Z/\ell^n \Z\) :  A ~(\text{ mod } \ell^{n-1}) \in \tilde{\rho}_{\ell^{n-1}}(\G) \}.	
\end{split}
\end{equation}
If $A \in \GL_2\(\Zln\)$ satisfies $A~(\mod \ell^{n-1}) \in D_{\ell^{n-1}, q-1}$ 
and $\Psi_{q}(\tr(A)^2, \det(A)) = 0$, then from \eqref{imglnlift}, we get 
$A \in  \tilde{\rho}_{\ell^{n}}(\G)$ and  hence $A \in D_{\ell^n, q-1}$, 
since $\Psi_{q}(\tr(A)^2, \det(A)) = 0$.

Let $\cR_n$ denote the natural group homomorphism
$$
\cR_n ~:~ M_2(\Z) ~\rightarrow~  M_2\(\Z/\ell^n\Z\)
$$
given by the reduction modulo $\ell^n$.
Now fix a matrix
\begin{equation*}
A_0 ~=~ 
	\begin{pmatrix}
	a & b \\
	c & d 
	\end{pmatrix}
\in M_2(\Z)
\end{equation*}
such that $\cR_{n-1}(A_0) \in D_{\ell^{n-1}, q-1}$.
Let $B_\ell = \{0,1, \cdots , \ell-1\}$ and for any $\vec{v}=(x,y,z,w) \in B^4_{\ell}$, set
\begin{equation*}
	A_0(\vec{v}) ~=~ 
	\begin{pmatrix}
		a+\ell^{n-1}x & b+\ell^{n-1}y \\
		c+\ell^{n-1}z & d+\ell^{n-1}w
	\end{pmatrix}
	\in M_2(\Z).
\end{equation*}
Also let $L_{\ell^n, q-1}(A_0)$ denote the set of all quadruples 
$\vec{v} \in  B^4_{\ell}$ such that
\begin{equation*}
\Psi_{q}(\tr(A_0(\vec{v}))^2,~ \det(A_0(\vec{v}))) ~\equiv~ 0 ~(\mod \ell^n).
\end{equation*}
Set $r = \tr(A_0), s = \det(A_0)$ and $A = A_0(\vec{v})$.
Then we have $\Psi_q(r^2, s) \equiv 0 ~(\mod \ell^{n-1})$ and
\begin{equation*}
\Psi_q(\tr(A)^2, \det(A)) 
~\equiv~ 
\Psi_q\(r^2 + 2r(x+w)\ell^{n-1}, ~s+(aw+dx-cy-bz)\ell^{n-1}\)
~\equiv~
0 ~(\mod \ell^n).
\end{equation*}
By Taylor's expansion of $\Psi_q$ at $(r^2, s)$, we have
\begin{equation*}
\Psi_q(r^2,s) 
~+~ 2r(x+w) \frac{\pd \Psi_q}{\pd X}(r^2,s) \ell^{n-1} 
~+~ (aw+dx-cy-bz) \frac{\pd \Psi_q}{\pd Y}(r^2, s) \ell^{n-1} 
~\equiv~
 0 ~(\mod \ell^n).
\end{equation*}
This implies that
\begin{equation*}
\alpha_1 x ~+~ \alpha_2 w ~+~ \alpha_3 y ~+~ \alpha_4 z 
~\equiv~
 - \frac{\Psi_q(r^2,s)}{\ell^{n-1}} ~(\mod \ell),
\end{equation*}
where
\begin{equation*}
\begin{split}
& \alpha_1~=~ 2r \frac{\pd \Psi_q}{\pd X}(r^2,s) 
       ~+~ d \frac{\pd \Psi_q}{\pd Y}(r^2,s), 
\phantom{mm} 
\alpha_3 ~=~ -c \frac{\pd \Psi_q}{\pd Y}(r^2,s), \\
& 
\alpha_2 ~=~ 2r \frac{\pd \Psi_q}{\pd X}(r^2,s) 
~+~ a \frac{\pd \Psi_q}{\pd Y}(r^2,s),
\phantom{mm}
\alpha_4 ~=~ -b \frac{\pd \Psi_q}{\pd Y}(r^2,s).
\end{split}
\end{equation*}
Let us first assume that $\ell \neq q$.
We claim that one of 
$\alpha_1,\alpha_2,\alpha_3,\alpha_4$ is not divisible $\ell$. If not, then
$\alpha_1 \equiv \alpha_2 \equiv \alpha_3 \equiv \alpha_4 \equiv 0 ~(\mod \ell)$. 
Applying \lemref{pdPsi}, we then get
$$
b ~\equiv~ c ~\equiv~ 0 ~(\mod \ell) 
\phantom{mm} \tx{and}\phantom{mm} 
a ~\equiv~ d ~(\mod \ell).
$$
Hence we obtain
$$
\Psi_q(\tr(A_0)^2, ~\det(A_0)) 
~\equiv~
 \Psi_q(4a^2, a^2) 
 ~\equiv~
  a^{q-1}q ~\equiv~ 0 ~(\mod \ell).
$$
This implies that $a \equiv 0 ~(\mod \ell)$, 
a contradiction to the fact that 
$\cR_{n-1}(A_0) \in \GL_2(\Z/\ell^{n-1}\Z)$. 
This proves our claim. Hence we deduce that 
$|L_{\ell^n, q-1} (A_0)| = \ell^3$. 
Thus we conclude that $|D_{\ell^n, q-1}| \leq \ell^3 |D_{\ell^{n-1}, q-1}|$
and equality holds if $\ell$ is sufficiently large.

Now let us consider the case when $\ell = q$.  It follows from \lemref{vPsi}
that
$$
\Psi_q(X,Y) ~\equiv~ 0 ~(\mod q^2) 
\phantom{mm}\text{with}\phantom{m}
\tx{gcd}(X,Y,q)~=~ 1
$$ 
has no solution for $q \geq 5$. Hence $|D_{q^n, q-1}| ~=~ 0$ for $q \geq 5$.
Finally, when $q = 3$, arguing as before, we can show that $|L_{3^n,2}(A_0)| \leq 3^3$. 
Hence we obtain
$|D_{3^{n}, 2}| \leq 3^3 |D_{3^{n-1}, 2}|$. Now by applying induction on $n$, we conclude that
\begin{equation}\label{Dln2}
|D_{\ell^n, q-1}| \leq 
\begin{cases}
& \ell^{3(n-1)} |D_{\ell,q-1}|, \phantom{m}\tx{if }\phantom{m}\ell \neq q 
~~\tx{ or }~ \ell=q=3 \\
& 0, \phantom{mmmmmmm}\tx{if}\phantom{m} ~\ell = q \geq 5
\end{cases}
\end{equation}
and equality holds if $\ell$ is sufficiently large.
Now the Lemma follows applying \eqref{rhoq-1=}, \eqref{Tl}, \eqref{Tlall} and \eqref{Dln2}.
\end{proof}

\subsection{Proof of \thmref{thm1} }
Let $f$ be as in \thmref{thm1} and $q$ be an odd prime.
Set
\begin{equation*}
	S_q ~=~ S_{q, y} ~=~ 
	\left\{ p  ~:~ p \nmid N,~ a_f(p^{q-1}) \neq 0, ~ 
	P(a_f(p^{q-1})) \leq y(p) \right\},
\end{equation*}
where $y=y(x)$ is a large positive real number which will be chosen later.
Also let 
$$
S_q(x) =  \{ p\leq x : p \in S_q \}.
$$
Consider the prime factorization of the following product 
\begin{equation}\label{prodafp}
\prod_{p \in S_q(x)} |a_f(p^{q-1})| 
= \prod_{\ell ~\tx{prime}} {\ell}^{\nu_{x, \ell}},
\end{equation}
where
$$
\nu_{x, \ell} 
= \sum_{p \in S_q(x)} \nu_{\ell}(a_f(p^{q-1})).
$$
Using Deligne's bound, we have
\begin{eqnarray}\label{vxq}
\nu_{x,\ell} 
~\leq~ 
\sum_{\substack{p \leq x \\ a_f(p^{q-1}) \neq 0}} 
\nu_{\ell}(a_f(p^{q-1})) \nonumber
&=&
\sum_{\substack{p \leq x \\ a_f(p^{q-1}) \neq 0}} 
 ~~\sum_{\substack{m \geq 1 \\ {\ell}^m | a_f(p^{q-1})}} 1 \\
&=&
\sum_{1 \leq m \leq \frac{\log(q x^{q(k-1)/2})}{\log \ell}} 
\sum_{\substack{p \leq x \\ a_f(p^{q-1}) \neq 0 \\ 
		a_f(p^{q-1}) \equiv 0 (\mod \ell^m)}} 1  \nonumber \\
&=& 
\sum_{1 \leq m \leq \frac{\log(q x^{q(k-1)/2})}{\log \ell}} 
 \pi_{f,~q-1}^{*}(x, \ell^m).
\end{eqnarray}
We choose $c_0$ sufficiently small depending on $f$ so that
$\log x > c d^4 \log (dN)$ (where $c$ is as in \thmref{CTZ})
if $1< d \leq c_0 \frac{(\log x)^{1/4}}{(\log\log x)^{1/4}}$.
Then applying \thmref{CTZ}, we have 
\begin{equation}\label{unpixd}
	\pi_{f,~q-1}^*(x, d) ~\ll~ \d_{q-1}(d) \pi(x).
\end{equation}
Set
\begin{equation*}
	z= c_0 \frac{(\log x)^{1/4}}{(\log\log x)^{1/4}}
\end{equation*}
and $x$ be sufficiently large from now onwards .
Fix a prime $\ell$ such that $\nu_{x,\ell} \neq 0$ and so $\ell \leq y$. 
Set
$$
m_0 = m_0(x, \ell)=  \Big[\frac{\log z}{\log \ell}\Big].
$$ 
  We estimate the sum
 in \eqref{vxq} by dividing it into two parts; 
 for $1\le m \le m_0$ and then 
 for $m> m_0$.
 It follows from \lemref{deq-1l}, \lemref{dq-1lm} and \eqref{unpixd} 
 that the first sum is
\begin{eqnarray}\label{vxq1}
	\sum_{1 \leq m \leq m_0}  \pi_{f,~q-1}^{*}(x, \ell^m) 
	~\ll~ 
	\sum_{1\leq m \leq m_0} q \cdot \frac{ \pi(x)}{\ell^m} 
	~\ll~ 
	q \cdot \frac{\pi(x)}{\ell} .
\end{eqnarray}
The second sum is
\begin{eqnarray}\label{vxq2}
	\sum_{m_0 < m \leq \frac{\log(q x^{q(k-1)/2})}{\log \ell}} 
	\pi_{f,~q-1}^{*}(x, \ell^m)
	&\leq&
	\pi_{f,~q-1}^{*}(x, \ell^{m_0}) 
	\sum_{m_0 < m \leq \frac{\log(qx^{kq})}{\log \ell}} 1 \nonumber\\
	&\ll&
  \frac{q}{\ell^{m_0}}  \pi(x) \cdot\frac{q\log x}{\log \ell}
	\nonumber\\
	& \ll &
	\frac{q^2 x}{\ell^{m_0} \log \ell}
	~\ll~ 
	\frac{q^2 x}{z} \cdot~\frac{ \ell}{\log \ell}.
\end{eqnarray}
From \eqref{vxq}, \eqref{vxq1} and \eqref{vxq2}, we get
\begin{equation}\label{vxq3}
	\nu_{x,\ell} ~\ll~ \frac{q^2 x}{z} \cdot \frac{\ell}{\log \ell}.
\end{equation}
Note that if $\ell \not\equiv 0, \pm 1 ~(\mod q)$, then we have $C_{\ell^m, ~q-1} = \emptyset$ 
(see \lemref{deq-1l} and \lemref{dq-1lm}).
Hence if $\ell^m \mid a_f(p^{q-1})$, then we must have $p \mid \ell N$ (see \S 2.1). 
Further $p \in S_q$ gives us that $p = \ell$. Hence we have 
$\nu_{x, \ell} \leq \nu_{\ell}(a_f(\ell^{q-1})) \ll  kq$ if  $\ell \not\equiv 0, \pm 1 ~(\mod q)$.
It follows from \eqref{prodafp}  that
\begin{equation}\label{log=vxll}
	\sum_{p \in S_q(x)} \log|a_f(p^{q-1})| 
	~=~
	 \sum_{\ell \leq y} \nu_{x, \ell} \log \ell.
\end{equation}
Thus applying \eqref{vxq3} and Brun-Titchmarsh inequality (see \cite{HalR}, Theorem~3.8), we obtain
\begin{equation*}\label{upp}
\begin{split}
\sum_{\substack{\ell \leq y \\ \ell \equiv 0, \pm 1 (\mod q)}}
\nu_{x,\ell} \log \ell
&\leq \nu_{x,q} \log q ~~+  \sum_{\substack{\ell \leq y \\ \ell \equiv \pm 1 (\mod q)}}
\nu_{x,\ell} \log \ell \\
& ~\ll~
\frac{q^3 x}{z} ~~+~
\frac{q^2 x}{z}
\sum_{\substack{\ell \leq y \\\ell \equiv \pm 1 (\mod q)}} \ell 
~\ll~~ 
\frac{q^3 x}{z} ~+~ \frac{q^2 x}{z}  \frac{y^2}{q \log (y/q)} 
\end{split}
\end{equation*}
for all $x$ sufficiently large depending on $f, \e, q$.
Also we have
\begin{equation*}
\begin{split}
\sum_{\substack{\ell \leq y \\ \ell \not\equiv 0, \pm 1(\mod q)}}
 \nu_{x,\ell} ~\log \ell 
~~\ll~~
q \sum_{\ell \leq y} \log \ell 
~\ll~
q y.
\end{split}
\end{equation*}
Hence we conclude that
\begin{equation}\label{vxllog}
 \sum_{\ell \leq y} \nu_{x, \ell} \log \ell
 ~\ll~
 \frac{q^3 x}{z} ~+~ \frac{q x}{z}  \frac{y^2}{ \log (y/q)} 
\end{equation}
for all $x$ sufficiently large depending on $f, \e, q$ and the implied constant depends only on $f$. 

Let $c_2$ be a positive constant such that $q-1 -c_1 \log (q-1) \geq q/2$ 
whenever $q \geq c_2$. Here $c_1$ is a constant as in \lemref{Lafpm}.
Hence for  $q \geq c_2$, by applying \lemref{Lafpm}, we get
\begin{equation}\label{logapg}
	\sum_{p \in S_q(x)} \log|a_f(p^{q-1})| 
	~\gg~
	 q \sum_{p \in S_q(x)} \log p.
\end{equation}
By partial summation, we have
\begin{equation}\label{par}
\sum_{p \in S_q(x)} \log p 
~=~
\# S_q(x) \log x ~+~ O\(\pi(x)\).
\end{equation}
Let $ \e > 0$ be a real number and  $y=y(x) = (\log x)^{1/8} (\log\log x)^{3/8 - \epsilon}$ from now on.
Then from \eqref{log=vxll}, \eqref{vxllog}, \eqref{logapg} and \eqref{par}, we get
\begin{equation*}
	\# S_q(x) ~\ll~ \frac{\pi(x)}{(\log\log x)^{2\e}}
\end{equation*}
for all $x$ sufficiently large depending on $f, \e, q$. This completes the proof of the theorem for $q \geq c_2$ by using \lemref{apn0}.

Now suppose that $q \leq c_2$. By Deligne's bound, we can write
$$
a_f(p) ~=~ 2 p^{\frac{k-1}{2}} \cos \theta_p~,
\phantom{m} 
\theta_p \in [0, \pi].
$$
For any prime $p \nmid N$, by using \eqref{PhPs}, we get
\begin{equation}\label{afPsi}
\begin{split}
a_f(p^{q-1}) ~=~ \Psi_q(a_f(p)^2, p^{k-1}) 
&~=~ \prod_{j=1}^{\frac{q-1}{2}} 
\( a_f(p)^2 - 4 \cos^2(\pi j/q) p^{k-1}\) \\
&~=~ (4 p^{k-1})^{\frac{q-1}{2}}  \prod_{j=1}^{\frac{q-1}{2}}
 \(\cos^2 \theta_p - \cos^2(\pi j/q)\).
\end{split}
\end{equation}
Set $\eta(x) =  \frac{\log\log x}{\sqrt{\log x}}$. 
For $1 \le j \le (q-1)/2$, let
$
\cJ_{x,j,q} =
 \left[ \cos(\pi j/q) - \eta(x), ~ \cos(\pi j/q) + \eta(x)\right]$ 
and 
$ 
 \cJ_{x,j,q}^{-} 
 = \left[ -\cos(\pi j/q) - \eta(x), ~ -\cos(\pi j/q) + \eta(x)\right]
$. 
Also set
\begin{equation*}
	\cI_{x,q} ~=~ [-1,~1]\setminus \(\bigcup_{1 \leq j \leq \frac{q-1}{2}} \(\cJ_{x,j,q} \cup \cJ_{x,j,q}^{-}\)\).
\end{equation*}
Then from \eqref{afPsi}, we get 
\begin{equation}\label{Ieta}
\begin{split}
\#\left\{ p \leq x ~:~  p \nmid N,~ \lambda_f(p) ~\in~ \cI_{x,q}\right\} 
&~\leq~
\#\left\{ p \leq x ~:~ p \nmid N,~ 
|a_f(p^{q-1})| 
~\geq~
(4p^{k-1})^{\frac{q-1}{2}} {\eta(x)}^{q-1} \right\}\\
&
~\leq ~\#\left\{ p \leq x ~:~ p \nmid N,~ 
|a_f(p^{q-1})| 
~\geq~
(4p^{k-1})^{\frac{q-1}{2}} {\eta(p)}^q \right\} 
~+~ O(\sqrt{x}),
\end{split}
\end{equation}
where the last inequality follows from the fact that 
if $x$ is sufficiently large and $p \in [\sqrt{x}, x]$, 
then we have ${\eta(p)}^q \leq {\eta(x)}^{q-1}$.
By \thmref{EffST}, we deduce that
\begin{equation}\label{lambdI}
\#\left\{ p \leq x ~:~  p \nmid N,~ \lambda_f(p) \in \cI_{x,q}\right\} 
~=~
\pi(x) ~+~ O\(\pi(x) \eta(x)\).
\end{equation}
 Set 
 $T 
 = \left\{ p ~:~ p \nmid N, ~ |a_f(p^{q-1})| 
 ~\geq~   (4p^{k-1})^{\frac{q-1}{2}} {\eta(p)}^q   \right\}$ 
and 
$T(x) = \{p \leq x ~:~ p \in T\}$. Also set $W = S_{q} \cap T$ and
$W(x)= \{ p\le x ~:~ p \in W \}$. 
Then from \eqref{Ieta} and \eqref{lambdI}, we get
$$
\# W(x) ~=~ \#S_{q}(x) ~+~ \#T(x) ~-~ \#(S_{q}\cup T)(x) 
		~\geq~ \#S_{q}(x) ~+~ O\(\pi(x) \eta(x)\).
$$
Thus we have
\begin{equation}\label{Lpwx0}
\sum_{p \in S_{q}(x)} \log|a_f(p^{q-1})| 
~\geq ~
\sum_{p \in W(x)} \log|a_f(p^{q-1})| 
~\geq~
\frac{(q-1)(k-1)}{2} \sum_{p \in W(x)} \log p ~+~ O\(\pi(x) \log\log x\).
\end{equation}
By partial summation, we get
\begin{equation}\label{LpWx}
\sum_{p \in W(x)} \log p 
~=~ \#W(x) \log x ~+~ O\(\pi(x)\) 
~\geq~ 
\#S_{q}(x)\log x ~+~ O\(x \eta(x)\). 
\end{equation}
From \eqref{log=vxll}, \eqref{vxllog}, \eqref{Lpwx0} and \eqref{LpWx}, we conclude that
\begin{equation*}
 \#S_q(x)  ~\ll~ \frac{\pi(x)}{(\log\log x)^{2\e}}
\end{equation*}
for all $q \leq c_2$ and $x$ sufficiently large depending on $f, \e$.
This completes the proof of the first part of \thmref{thm1}. 

Set $\tilde{S}_q = \{p : p \nmid N, ~ a_f(p^{q-1}) \neq 0,~ 
P(a_f(p^{q-1})) \leq q^{-\e} (\log p)^{1/8} (\log\log p)^{3/8}\}$.
Also set
$\tilde{S}_q (x) = \{ p \le x ~:~ p \in \tilde{S}_q \}$, ~~
$z = c_0 (\log x)^{1/4} (\log\log x)^{-1/4}$ 
and 
$
y = q^{-\e} (\log x)^{1/8} (\log\log x)^{3/8}.
$
Arguing as before, we can deduce that
$$
\# \tilde{S}_q(x) ~\ll~ \frac{\pi(x)}{q^{2\e}}
$$
for all $x$ sufficiently large depending on $f, \e, q$ and 
the implied constant depends only on $f$.
Thus 
$$
\limsup_{x \to \infty} \frac{\#\tilde{S}_q(x)}{\pi(x)}  ~<~ 1
$$
 if $q$ is sufficiently large. 
This completes the proof of the second part of \thmref{thm1}.

\begin{rmk}
We note that for any real valued non-negative function $u$ 
satisfying $u(x) \to 0$ as $x \to \infty$, the lower bound 
$(\log p)^{1/8}(\log\log p)^{3/8-\e}$ in \thmref{thm1} can be replaced 
with $(\log p)^{1/8}(\log\log p)^{3/8}u(p)$ and also the term $q^{-\e}$ can be replaced with $u(q)$.
\end{rmk}

\subsection{Proof of \rmkref{rmkthm2-1}}
For any prime $p$ and integer $n \geq 1$, the function $\tau$ satisfies 
\begin{equation*}
	\tau(p^{n+1}) ~=~ \tau(p) \tau(p^n)- p^{11} \tau(p^{n-1}).
\end{equation*} 
Hence we get
\begin{equation}\label{Lsaf}
	\tau(p^{n-1}) ~=~ \frac{\a_p^n-\b_p^n}{\a_p-\b_p},
\end{equation}
where $\a_p, \b_p$ are roots of the polynomial 
$x^2 - \tau(p)x + p^{11}$.
From \eqref{Lsaf} and the fact that $\tau$ takes integer values, 
it follows that  
$$
\tau(p^{d-1}) ~\div~ \tau(p^{2n})
\phantom{mm} \text{whenever} \phantom{m}
d ~\div~ 2n+1
$$ 
(see page 37, Theorem IV of \cite{RC} and page 434, Eq. 14 of \cite{CS}).
Hence we get 
\begin{equation*}
	P(\tau(p^{2n})) ~\geq~ P(\tau(p^{d-1})).
\end{equation*}
Now \thmref{thm0} follows from \thmref{thm1} by choosing 
$d$ to be the largest prime factor of $2n+1$. More generally, we can 
deduce \thmref{thm0} for any non-CM cuspidal Hecke eigenform $f$ of 
weight $k$ and level $N$ with integer Fourier coefficients 
by noting that if $d \mid 2n+1$, then we have
$$
P(a_f(p^{2n})) ~\geq~ P(a_f(p^{d-1}))
$$
for any prime $p \nmid N$ with $a_f(p^{2n}) \neq 0$.

\medskip

\subsection{Proof of \corref{Corpm}} Note that if $\tau(p) \neq 0$, then $\tau(p) \mid \tau(p^{2n-1})$ and hence we have $P\(\tau(p^{2n-1})\) \geq P\(\tau(p)\)$ (see Section 5.3). Now \corref{Corpm} follows from \eqref{low-1/8} and \thmref{thm0}.

\subsection{Proof of \thmref{thm2}}
The proof follows along the lines of the proof of \thmref{thm1}.
 For any real number $0 < c <1$, let
\begin{equation*}
	S_{q,c} ~=~ \left\{p  ~:~ p \nmid N,~ a_f(p^{q-1}) \neq 0, ~
	 P(a_f(p^{q-1})) ~\leq~ c p^{1/14}(\log p)^{2/7}\right\}.	 
\end{equation*}
Also let $S_{q,c}(x) = \{  p \leq x ~:~ p \in S_{q,c} \}$. We will show that if $c$ is sufficiently small, then
\begin{equation*}\label{limsup}
	\limsup_{x \rightarrow \infty} \frac{\#S_{q,c}(x)}{\pi(x)} 
	~\leq~ 
	\frac{2}{13(k-1)}.
\end{equation*}
Set
\begin{equation*}
y ~=~ c x^{1/14}(\log x)^{2/7} 
\phantom{mm}\text{and}\phantom{mm} 
z ~=~ c\frac{x^{1/7}}{(\log x)^{3/7}}.
\end{equation*}
From the prime factorization
\begin{equation*}\label{c1}
	\prod_{p \in S_{q,c}(x)} |a_f(p^{q-1})| 
	~=~ \prod_{\ell \leq y} \ell^{\nu_{x,\ell}},
\end{equation*}
we get
\begin{equation}\label{l=vxl}
	\sum_{p \in S_{q,c}(x)} \log|a_f(p^{q-1})| 
	~=~ \sum_{\ell \leq y} \nu_{x,\ell} \log \ell.
\end{equation}
As before, we have
\begin{equation*}
	\nu_{x, \ell} 
	~\leq~
	 \sum_{1 \leq m \leq \frac{\log(q x^{qk})}{\log \ell}}  \pi_{f,~q-1}^{*}(x, \ell^m).
\end{equation*}
 Fix a prime $\ell \leq y$ such that $\nu_{x, \ell} \neq 0$. 
 If $\ell \not\equiv 0, \pm 1 ~(\mod q)$, then as before, we have
 $$
 \nu_{x, \ell} \leq \nu_{\ell} (a_f(\ell^{q-1})) = O(kq).
 $$
Now suppose that $\ell \equiv 0, \pm 1 (\mod q)$ and set  $m_0 = \Big[\frac{\log z}{\log \ell}\Big]$.
Let $x$ be sufficiently large from now on.
Then applying \thmref{GRHpifmd}, \lemref{deq-1l} and \lemref{dq-1lm}, we get 
\begin{equation}\label{pif1}
\begin{split}
\sum_{1 \leq m \leq m_0}  
\pi_{f,~q-1}^{*}(x, \ell^m) 
&= \sum_{1 \leq m \leq m_0} 
\bigg\{ \d_{q-1}(\ell^m)\pi(x) 
~+~ 
O\( \d_{q-1}(\ell^m) \ell^{4m} x^{1/2} \log x \)\bigg\}\\
& ~\leq~
\frac{q-1}{2}\frac{\pi(x)}{\ell} 
~+~ 
O\(\frac{ q\pi(x)}{\ell^2}\) 
~+~ 
O\( q z^3 x^{1/2} \log x\).
\end{split}
\end{equation}
Again applying \thmref{GRHpifmd}, \lemref{deq-1l} and \lemref{dq-1lm}, we obtain
\begin{equation}\label{pif2}
\begin{split}
\sum_{m_0 < m \leq \frac{\log(q x^{qk})}{\log \ell}}  
\pi_{f,~ q-1}^{*}(x, \ell^m) 
&\leq~ 
\pi_{f,~ q-1}^{*}(x, \ell^{m_0}) 
\sum_{m_0 \leq m \leq \frac{\log(q x^{qk})}{\log \ell}} 1\\
& \ll~ 
\(\frac{ q \pi(x)}{\ell^{m_0}} 
~+~ 
q ~\ell^{3m_0} x^{1/2} \log x\)  \frac{q \log x}{\log \ell}\\
&\ll~
\frac{q^2 x}{z}\frac{\ell}{\log \ell} 
~+~ 
q^2 z^3 x^{1/2}\frac{(\log x)^2}{\log \ell}.
\end{split}
\end{equation}
From \eqref{pif1} and \eqref{pif2}, we get
\begin{equation}\label{pif3}
\nu_{x,\ell} 
~\leq~
\frac{q-1}{2}\frac{\pi(x)}{\ell}  
~+~ O\(\frac{ q\pi(x)}{\ell^2}  
~+~  \frac{q^2 x}{z}\frac{\ell}{\log \ell}
~+~  q^2 z^3 x^{1/2}\frac{(\log x)^2}{\log \ell}\).
\end{equation}
It follows from \eqref{pif3} and Brun-Titchmarsh inequality that
\begin{equation}\label{1pm1q}
\begin{split}
\sum_{\substack{\ell \leq y \\ \ell \equiv \pm 1(\mod q)}} 
\nu_{x,\ell} \log \ell 
~\leq~  
\pi(x) \log y 
+ c_2\( q \pi(x) + \frac{q x}{z}\frac{y^2}{\log (y/q)} 
+  q z^3 x^{1/2} (\log x)^2 \frac{y}{\log (y/q)}\)
\end{split}
\end{equation}
for all sufficiently large $x$ (depending on $f,q$). 
Here $c_2$ is a positive constant depending only on~$f$. We also have
\begin{equation}\label{vxqn01}
\begin{split}
\nu_{x,q} \log q 
~\ll~
  \pi(x) \log q~+~ \frac{q^3 x}{z} 
 ~+~ q^2 z^3 x^{1/2} (\log x)^2 \phantom{mm}\tx{and}\phantom{m}
\sum_{\substack{\ell \leq y \\ \ell \not\equiv 0, \pm 1 (\mod q)}} 
\nu_{x, \ell}\log \ell  
~\ll~ q y.
\end{split}
\end{equation}
Let us fix $c$ such that $11111 \cdot c_2 (c+c^4) < 1$. 
Then by substituting the values for $y$ and $z$ in \eqref{1pm1q} and \eqref{vxqn01}, we deduce that
\begin{equation}\label{vx113.5}
\sum_{\ell \leq y} \nu_{x, \ell} \log \ell 
~~<~ \frac{q-1}{13.5}~ x 
\end{equation}
for all sufficiently large $x$ (depending on $f,q$).

Let $q$ be  sufficiently large such that 
$q-1-c_1 \log (q-1)> 13 q/13.4$, then
from \lemref{Lafpm}, we get
\begin{equation}\label{lvx1}
\sum_{p \in S_{q,c}(x)} \log |a_f(p^{q-1})|
~\geq~
 \frac{13 q}{13.4}\frac{k-1}{2} 
 \sum_{p \in S_{q,c}(x)} \log p.
\end{equation}
 By partial summation, we have 
\begin{equation}\label{lvx2}
\sum_{p \in S_{q,c}(x)} \log p 
~=~ \#S_{q,c}(x)\log x + O \(\pi(x)\).
\end{equation}
Hence from \eqref{l=vxl}, \eqref{vx113.5}, \eqref{lvx1} and \eqref{lvx2}, we conclude that 
\begin{equation*}
	\# S_{q,c}(x) ~\leq~ \frac{2}{13 (k-1)} \pi(x)
\end{equation*}
for all sufficiently large $x$ depending on $f, q$. This completes 
the proof of first part of \thmref{thm2} for sufficiently large $q$. For small values of $q$, 
as in \S 5.2, we can show that
\begin{equation}\label{lSq}
\sum_{p \in S_{q,c}(x)} 
\log |a_f(p^{q-1})|
~\geq~
\frac{(q-1)(k-1)}{2} \#S_{q,c}(x) \log x 
~+~ O\( x \frac{\log\log x}{\sqrt{\log x}}\).
\end{equation} 
Hence from \eqref{l=vxl}, \eqref{vx113.5} and \eqref{lSq}, we deduce that
\begin{equation*}
	\# S_{q,c}(x) ~\leq~ \frac{2}{13 (k-1)} \pi(x)
\end{equation*}
for all sufficiently large $x$ depending on $f$. This completes the proof of first part of \thmref{thm2}.

\medskip

The proof of second part follows along the lines of the proof of first part of \thmref{thm2}. Without loss of generality, 
we can assume that $g(x)\log x \to \infty$ as $x \to \infty$, 
by replacing the function $g(x)$ by $g(x)+ \frac{1}{\log\log x}, ~x \geq 3$ if necessary.
Set
$$
S_q ~=~ \left\{p  ~:~  p\nmid N, ~ a_f(p^{q-1}) \neq 0, P(a_f(p^{q-1})) ~\leq~ p^{g(p)}  \right\}
$$
and $S_q(x) = \{p \leq x ~:~ p \in S_q\}$. Also set $h(x) = \max\left\{g(t) : \sqrt{x} \leq t \leq x\right\}, 
~y = x^{h(x)}$ and $z=y^2$. From the following prime factorization
\begin{equation*}
\prod_{\substack{p \in S_q(x) \\ p> \sqrt{x}}} |a_f(p^{q-1})| 
~=~
 \prod_{\ell \leq y} {\ell}^{\nu_{x, \ell}}~,
\end{equation*}
we get
\begin{equation*}
	\sum_{\substack{p \in S_{q}(x) \\ p > \sqrt{x}}} \log|a_f(p^{q-1})| 
	~=~ 
	\sum_{\ell \leq y} \nu_{x,\ell} \log \ell.
\end{equation*}
Now proceeding as before, we can show that
$$
\sum_{\ell \leq y} \nu_{x,\ell} \log \ell  
~\ll~
 x h(x) ~+~ q \pi(x) ~+~q \frac{x}{h(x) \log x}+ q^3 \frac{x}{e^{2 h(x)\log x}}
$$
and 
$$
\sum_{p \in S_{q}(x)} \log|a_f(p^{q-1})| 
~\gg~ q \#S_q(x)\log x ~+~ O\(qx \eta(x)\),
$$
where $\eta(x) = \log\log x / \sqrt{\log x}$. Also note that
 $$
 \sum_{\substack{p \in S_{q}(x) \\ p \leq \sqrt{x}}} \log|a_f(p^{q-1})| 
~\ll~ q \sqrt{x}.
$$
Hence we deduce that
$$
\#S_q(x) 
~\ll~ \pi(x)h(x) 
~+~ \frac{\pi(x) \log\log x}{\sqrt{\log x}} 
~+~ \frac{\pi(x)}{h(x)\log x} 
~+~ q^2 \frac{\pi(x)}{e^{2 h(x)\log x}}.
$$
This implies that
$$
\lim_{x \to \infty} \frac{\#S_q(x)}{\pi(x)} ~=~ 0.
$$
This completes the proof of \thmref{thm2}.

\section*{Acknowledgments}
The authors would like to thank Purusottam Rath for his comments on an earlier version
of the article and would like to thank the referee for suggesting an alternate proof of 
Lemma 15 and other valuable suggestions.
Also the authors would like to acknowledge the support of DAE number
theory plan project.

\end{document}